\documentclass{amsart}
\usepackage[utf8]{inputenc}
\usepackage{amsmath}
\usepackage{amssymb}
\usepackage{caption}
\usepackage{amsthm}
\usepackage{mathtools}
\usepackage[dvipsnames]{xcolor}
\usepackage{xcolor}
\usepackage{tabularray}

\usepackage[backref=false, colorlinks, linkcolor=blue, citecolor=Blue]{hyperref}
\usepackage{cleveref}
\usepackage{nicefrac}
\usepackage[inline]{enumitem}
\usepackage{enumitem}
\usepackage[]{todonotes}
\usepackage{amssymb}
\usepackage{mathrsfs}

\usepackage{array,float}
\usepackage{array}
\usepackage{multirow}
\usepackage{makecell}
\usepackage{colortbl}
\usepackage{tikz}
\usetikzlibrary{shapes,arrows}
\usetikzlibrary{fit,positioning}
\usetikzlibrary{patterns,decorations.pathreplacing}
\usetikzlibrary{tikzmark,decorations.pathreplacing,calc}
\usepackage{xkeyval}
\usepackage{moreverb}
\usepackage{epic}
\usepgfmodule{shapes,plot,decorations}

\usepackage{booktabs} 
\usepackage[normalem]{ulem} 

\newcommand{\N}{\mathbb{N}}
\newcommand{\Z}{\mathbb{Z}}
\newcommand{\Q}{\mathbb{Q}}
\newcommand{\R}{\mathbb{R}}
\newcommand{\C}{\mathbb{C}}

\newcommand{\HH}{\mathbb{H}}

\newcommand{\PGL}{\mathbf{PGL}}
\newcommand{\SL}{\mathbf{SL}}
\newcommand{\PSL}{PSL}
\newcommand{\SO}{\mathrm{SO}}
\newcommand{\GL}{\mathbf{GL}}

\newcommand{\vol}{\text{vol}}

\newcommand{\cO}{\mathcal{O}}

\newcommand{\hhat}{\widehat{h}}

\newcommand{\id}{\mathrm{Id}}

\newcommand{\covol}{\mathrm{covol}}

\newcommand{\ram}{\textrm{Ram}}

\newcommand{\tr}{\mathrm{tr}}

\newtheorem{theorem}{Theorem}[section]
\newtheorem{corollary}[theorem]{Corollary}
\newtheorem{lemma}[theorem]{Lemma}

\newtheorem{definition}[theorem]{Definition}

\newtheorem{question}[theorem]{Question}

\theoremstyle{remark}
\newtheorem{remark}[theorem]{Remark}

\theoremstyle{remark}
\newtheorem{example}[theorem]{Example}

\renewcommand{\qed}{\hfill$\scriptstyle\blacksquare$}

\title{A strong height gap theorem for $\PGL_2$}
\author[M. Belolipetsky]{Mikhail Belolipetsky}
\address{IMPA, Estrada Dona Castorina 110, 22460-320 Rio de Janeiro, Brazil}
\email[]{mbel@impa.br}

\author[S. Hurtado]{Sebastian Hurtado} 
\address{Department of Mathematics, Yale University, New Haven, CT 06511, USA}
\email[]{sebastian.hurtado-salazar@yale.edu}

\begin{document}

\begin{abstract}

The height gap theorem states that the finite subsets $F$ of matrices generating non-virtually solvable groups have normalized height $\hhat(F)$ bounded below by a constant. It was first proved by Breuillard and another proof was given later by Chen, Hurtado and Lee. In this paper we show that when the set $F$ is contained in a maximal arithmetic subgroup  $\Gamma$ of $G = \PGL_2(\R)^a\times\PGL_2(\C)^b$, $a+b \ge 1$, the height bound  for the case when $F$ generates a Zariski dense subgroup of $G$ over $\R$ is proportional to $\log(\covol(\Gamma))$, the function of the covolume of $\Gamma$. This result strengthens the theorem for the lattices of large covolume and has various applications including a strong version of the arithmetic Margulis lemma for $\PGL_2(\R)^a\times\PGL_2(\C)^b$. 

\end{abstract}

\maketitle

\section{Introduction}

The height gap theorem that was proved by Breuillard in \cite{Breu11} is a powerful tool in the study of lattices in linear groups. It can be thought of as a non-abelian analog of Lehmer's Mahler measure problem in number theory and as a global adelic analog of the Margulis Lemma in hyperbolic geometry. An alternative proof of Breuillard's theorem that uses the existence of almost laws on compact groups and is much more elementary was given by Chen, Hurtado and Lee in \cite{CHL21}. In \cite{FHR22}, Fraczyk, Hurtado and Raimbault showed that the height gap theorem implies a strengthening of the Margulis lemma for arithmetic subgroups.  This new arithmetic Margulis lemma has already found applications in solutions of open problems and new proofs of important results (cf. \cite{BCDT24}, \cite{FH23}, \cite{FHR22}).  

In this paper we prove a stronger version of the height gap theorem for subsets of arithmetic lattices in $G = \PGL_2(\R)^a\times\PGL_2(\C)^b$ and deduce from it a strong arithmetic Margulis lemma for $G$. The strengthening refers to relating the gap to the covolume of an arithmetic subgroup. These results have interesting applications some of which are considered in the last section of the paper.  

The idea behind our strong gap theorem is to apply the method of \cite{CHL21} to bound the height of a generic element of a group $\Gamma$ by a function of arithmetic invariants of $\Gamma$ and then use Borel's volume formula \cite{Bor81} to relate this bound to the covolume of a maximal lattice containing $\Gamma$. It is important for this argument to notice that generic elements capture a large amount of arithmetic information about the group, while the  uniform exponential growth for linear groups first established by Eskin, Mozes and Oh in \cite{EMO05} allows us to effectively detect these elements.

In order to state the main result we need to recall the notion of \emph{height}. Let $K$ be a number field and let $V_K$ be the set of absolute values on $K$ up to equivalence, which is a union of  archimedean places $V_\infty$ (corresponding to real and complex embeddings of $K$ into $\R$ or $\C$) and non-archimedean places $V_f$ (corresponding to prime ideals in the ring of integers $\cO_K$ of $K$). For $v\in V_K$, let $K_v$ be the corresponding completion of $K$ and define $n_v = [K_v: \Q_p]$, if  $v \vert p$, $n_v = 1$, if $K_v = \R$, and $n_v = 2$, if $K_v = \C$. For a vector $x = (x_1,x_2, \dots, x_m) $ in $K_v^m$, we define $\|x\|_{v} = \max_{i =1}^m |x_i|_v$ if $v$ is non-archimedean and $\|x\|_v = \sqrt{|x_1|_v^2 + \dots + |x_m|_v^2 }$ if $v$ is archimedean.

For a matrix $A$ in $\GL_m(K)$, let $\|A\|_v$ be the operator norm.  The \emph{height} of $A$ is defined as
$$h(A) = \frac{1}{[K:\Q]} \sum_{v \in V_K} n_v\log^{+} \|A\|_{v},$$ 
where $\log^{+}(x) = \log \max (|x|, 1)$ for a real number $x$. For a finite set $F$ of matrices in $\GL_m(K)$, the height of $F$ is defined as
$$h(F) = \max_{A \in F} h(A),$$ 
and the \emph{normalized height} is defined by 
$$\hhat (F) = \lim_{n \to \infty} \frac{h (F^n)}{n},$$ 
where $F^n$ consists of all products of $n$ elements in $F$.
The limit above exists because $h(F^n)$ is sub-additive, which also implies $h(F) \geq \hhat(F)$.

Our main result is the following strengthening of the height gap theorem of Breuillard \cite{Breu11, CHL21}:  

\begin{theorem}\label{thm1}
Let $G = \PGL_2(\R)^a\times\PGL_2(\C)^b$, $a+b \ge 1$. There exists a positive constant $c_G$ (depending only on $G$) with the following property. Consider an arithmetic lattice $\Gamma \subset G$ defined over a field $k$ and  let $F\subset \Gamma$ be a finite subset. Then either
\begin{enumerate}
\item The subgroup generated by $F$ is not Zariski dense in $G$ over $\R$, or
\item $\displaystyle\hhat(F) > c_G\max\left(\frac{\log(\covol(\Gamma_1))}{[k:\Q]^2}, 1\right)$, where $\Gamma_1$ is a congruence subgroup (e.g., a maximal lattice) containing $\Gamma$.
\end{enumerate}
\end{theorem}



\begin{remark}\label{rem1}
If $G = \PGL_2(\R)$, then a non Zariski dense subgroup is virtually solvable and the condition on $F$ can be stated in a more familiar form, i.e. here in part (1) we claim that $\langle F \rangle$ is virtually solvable or, equivalently for this case, a virtually abelian group. For the  other groups $G$ this implication is not always true and there are more possible options for the groups generated by $F$. 
A basic example is given by a sequence of Bianchi subgroups $\Gamma_D = \PSL_2(\Z[\sqrt{-D}])$, $D>0$, in $G = \PGL_2(\C)$. All the Bianchi groups contain the subgroup $\PSL_2(\Z)$ and they have the corresponding maximal lattices $\Gamma_{1,D}$ with $\covol(\Gamma_{1,D}) \to \infty$ as $D\to\infty$. The Zariski closure of $\PSL_2(\Z)$ in $G$ over $\R$ is $\PGL_2(\R)$, hence here we can choose the sets $F$ in (1) that generate not Zariski dense and not virtually solvable subgroups.
\end{remark}

\begin{remark}\label{rem2}
By means of some case by case considerations it is possible to deduce more precise information about the subgroup generated by $F$. For example, let $G = \PGL_2(\C)$ and $H$ is a maximal proper real Zariski closed subgroup of $G$. Considering the real subalgebras of the Lie algebra of $G$ we quickly deduce that $H$ is conjugate to the Borel subgroup of the upper triangular matrices, or to the maximal compact subgroup $\SO_3$, or to the real Lie group $\PGL_2(\R)$. In the first case the discrete subgroups of $H$ are virtually abelian, in the second case they are finite, and in the third case they are Fuchisan groups. Moreover, the classification of arithmetic subgroups of $\PGL_2(\C)$ implies that the latter case is possible only for arithmetic subgroups of type~I (i.e. those arithmetic subgroups which are defined by quadratic forms, see \cite[Theorem~10.2.3]{MR-book} and \cite{BBKS23}). Therefore, for type~II and type~III arithmetic subgroups the set $F$ in part (1) of the alternative for $G = \PGL_2(\C)$ generates a virtually abelian group.
\end{remark}

This discussion suggests an equivalent formulation of the gap theorem which in some cases will be more convenient:

\begin{theorem}\label{thm2}
Under the assumptions of Theorem~\ref{thm1} we have the following alternative:
\begin{enumerate}
\item The subgroup $\langle F \rangle$ is virtually solvable, or
\item The Zariski closure $H$ of $\langle F \rangle$ in $G$ over $\R$ is semi-simple, $\Gamma' = \Gamma\cap H$ is an arithmetic lattice in $H$ defined over a field $l$, and for any congruence subgroup $\Gamma_1$ of $H$ containing $\Gamma'$ we have
$$\displaystyle\hhat(F) > c_G\max\left(\frac{\log(\covol(\Gamma_1))}{[l:\Q]^2}, 1\right).$$
\end{enumerate}

\end{theorem}

\begin{remark} \label{rem:bounded deg}
If the degree of the field $k$ is bounded, then the estimate in part (2) of the theorems improves to 
$$\hhat(F) > c\log(\covol(\Gamma_1)),$$
with $c = c(G) > 0$. 
For instance, the improved estimate holds for non-uniform arithmetic lattices in $G$. 
\end{remark}

As a corollary of the main result we have the following strengthening of the arithmetic Margulis lemma:

\begin{corollary}[Strong arithmetic Margulis lemma]\label{cor1}
Let $G$ be a simi-simple Lie group $\PGL_2(\R)^a\times\PGL_2(\C)^b$, $a+b \ge 1$, and let $X$ be the symmetric space of $G$. There exists a positive constant $\epsilon_G$ with the following property. Consider an arithmetic subgroup $\Gamma \subset G$ and a point $x \in X$. Let $\Gamma_1$ be a congruence arithmetic subgroup (e.g., a maximal lattice) containing $\Gamma$. Then, the subgroup generated by the set
$$F = \{\gamma\in\Gamma \mid d(x, \gamma x) \leq \epsilon_G\log(\covol(\Gamma_1))^{\frac12}\}$$
is not Zariski dense in $G$ considered as a real algebraic group. 
\end{corollary}

In Section~\ref{sec:appls} we present two applications of Corollary~\ref{cor1} that give short new solutions to well-known problems.

\subsection*{Further remarks and questions}

Even though the inequality for $\hhat(F)$ in Theorem~\ref{thm1} is an improvement of Breuillard's original gap, it is not clear if this gap is optimal, and it might well be possible that the following estimate holds:  

\begin{question}\label{optimal} Can the inequality in Theorem~\ref{thm1}(2) be improved to 
$$\displaystyle\hhat(F) > c_G\frac{\log(\covol(\Gamma_1))}{[k:\Q]}\text{?}$$
\end{question}

If the improved bound holds, then Corollary \ref{cor1} would apply to the sets of the form $$F = \{\gamma\in\Gamma \mid d(x, \gamma x) \leq \epsilon_G\log(\covol(\Gamma_1)\},$$ 
which is optimal up to the value of the multiplicative constant $\epsilon_G$
due to geometry of non-compact symmetric spaces  (because the volume of balls grows at most exponentially). This observation also shows that if we have a uniform bound on the degree $[k:\Q]$, our height gap is optimal.

\subsubsection*{More general Lie groups} 
We expect that the proof of Theorem~\ref{thm1} can be generalized to other semi-simple real (and $S$-adic) Lie groups. The covolume is controlled  via Prasad's volume formula but there are many more possibilities for the groups and parahorics involved which have to be treated consistently. Generalizing the result to this setting is important, as it  will give an improvement of Breuillard's gap valid for any finite set of matrices in $\GL_n(\bar{\Q})$. 


\medskip

\noindent
\textbf{Acknowledgments.} 
We would like to thank Mikolaj Fraczyk who initially took part in this project and contributed significatly to our work. We also thank M.~Fraczyk, B.~Lowe and J.~Raimbault for various discussions of potential further applications of these results. The work of M.B. is partially supported by the FAPERJ grant E-26/204.250/2024,  the Institut Henri Poincar\'e (UAR 839 CNRS-Sorbonne Universit\'e), and LabEx CARMIN (ANR-10-LABX-59-01).

\section{Arithmetic preliminaries}



Borel \cite{Bor81} described in detail the maximal arithmetic lattices in $G = \PGL_2(\R)^a\times\PGL_2(\C)^b$ and gave a formula for their covolume. We will briefly recall these results following Borel's paper and the exposition of his work in \cite{MR-book}.

Arithmetic lattices in $G$ are all obtained in the following way. Let $k$ be a number field of degree $d=r_1+2r_2$ with $r_2 = b$ complex places and $r_1 \geq a$ real places. Let $A$ be a quaternion algebra over $k$. Assume that for all except fixed $a$ real places $v\in V_\infty$ of $k$ the algebra $A$ ramifies over $k_v (\cong \R)$, i.e., $A(k_v)$ is isomorphic to the Hamiltonian quaternion algebra, while $A(k_{v_i})\cong {\mathrm M}_2(\R)$ for the remaining real places $v_1$, \ldots, $v_a$. Let $\cO=\cO_k$ be the ring of integers in $k$ and let $\mathfrak{D}$ be an order in $A(k)$, i.e., $\mathfrak{D}$ is a finitely generated $\cO$-submodule of $A(k)$ which is also a subring that generates $A(k)$ over $k$. Furthermore, assume that $\mathfrak{D}$ is a maximal order in $A$. Let $\mathfrak{D}^*$ be the group of invertible elements of $\mathfrak{D}$. Now, $\mathfrak{D}^*$ is discrete in $\prod_{v\in V_{\infty}}
(A(k_v))^*\cong \mathrm{U}(2)^{r_1-a}\times \GL_2(\R)^a\times\GL_2(\C)^b$, and its projection 
$$\Gamma = P\rho(\mathfrak{D}^*) \subset \PGL_2(\R)^a\times\PGL_2(\C)^b$$ is a discrete subgroup of $G$. Any subgroup of $G$ which is commensurable with such $\Gamma$ is called an \emph{arithmetic subgroup}. It is always a \emph{lattice}, i.e., a discrete subgroup of finite covolume in $G$. 
All the arithmetic lattices of $G$ are obtained in this way and two arithmetic
lattices of $G$ are commensurable if and only if they come, in the process as above, from the same algebra $A$. So for each such algebra $A$ we can associate a well defined commensurability class of arithmetic lattices in $G$ which is denoted by $\mathcal{C}(A)$.

An arithmetic subgroup $\Gamma = P\rho(\mathfrak{D}^*)$ and its finite index subgroups are said to be \emph{derived from a quaternion algebra}. These lattices have a property that they are contained in $k$-points of the associated algebraic $k$-group. A subgroup $\Lambda^{(2)}$ generated by the squares of the elements of an arithmetic lattice $\Lambda$ is always derived from a quaternion algebra (cf. \cite[Corollary~8.3.5]{MR-book}). 

Now, a quaternion algebra $A$ over $k$ is completely determined by the finite set of valuations $\ram(A)$, a subset of the set of all valuations $V$ of $k$ which consists of those $v$ for which $A(k_v)$ \emph{ramifies} (i.e., $A(k_v)$ is a division algebra), while for $v\in V\setminus \ram(A)$, $A(k_v)$ \emph{splits} (i.e., isomorphic to ${\mathrm M}_2(k_v)$). The subset $\ram(A)$ must be of even size; in our case it is formed by all except $a$ real valuations and a subset $\ram_f(A)$, possibly empty, of non-archimedean valuations. The set $\ram_f(A)$ can be identified with a subset of the prime ideals of $\cO$. Let $\Delta (A)= \prod_{\mathcal{P}\in \ram_f(A)}\mathcal{P}$, the product of all prime ideals at which $A$ ramifies.

Borel showed that the commensurability class $\mathcal{C}(A)$ has infinitely many non-conjugate maximal elements whose covolumes form a discrete subset of $\R$. The minimal covolume in the class $\mathcal{C}(A)$ is:
\begin{equation}			\label{formula:maxcovolume}
	\covol(\Gamma_\mathfrak{D}) = \frac{2
		{\Delta_k}^{3/2}\zeta_k(2)\prod_{\mathcal{P}|\Delta(A)}(N(\mathcal{P})-1)}
		{2^{2r_1+3r_2-2a}\pi^{2r_1+2r_2-a}[\Gamma_\mathfrak{D}:\Gamma^1_\mathfrak{D}]},
\end{equation}
see \cite[Theorem~7.3 and Section~8]{Bor81} (see also \cite[Corollary 11.6.6, p. 361 and (11.6), p. 333]{MR-book} where the formula is given for $G = \PGL_2(\C)$).

Here $\Delta_k$ is the absolute value of the discriminant of $k$, $\zeta _k$ is the
Dedekind zeta function of $k$ and $N(\mathcal{P})$ denotes the norm of the ideal $\mathcal{P}$, i.e.
the order of the quotient field $\cO /\mathcal{P}$. The index of a principal arithmetic subgroup $\Gamma^1_\mathfrak{D}$ in the associated maximal arithmetic subgroup $\Gamma_\mathfrak{D}$ can be further estimated in terms of number theoretic invariants but we will only use the fact that it is greater or equal than $1$. 

Let $S$ be a finite subset of $V_f$ disjoint from $\ram_f(A)$. These subsets parametrize the maximal arithmetic subgroups $\Gamma_S$ in the commensurability class $\mathcal{C}(A)$: for $v\not\in S$ the group $\Gamma_S$ fixes a vertex of the Bruhat--Tits tree of $\SL_2(k_v)$ and for $v \in S$ the group $\Gamma_S$ stabilizes an edge of the Bruhat--Tits tree. It is shown that there exists a maximal arithmetic subgroup $\Gamma_S$ with these properties and that all maximal subgroups in $\mathcal{C}(A)$ are of this form (see \cite[Section~2]{Bor81}). The covolume of a maximal arithmetic subgroup $\Gamma_S$ is given by \eqref{formula:maxcovolume} together with the equation for the generalized index 
\begin{equation}\label{formula:covolume}
[\Gamma_\mathfrak{D}:\Gamma_S] = 2^{-m} \prod_{\mathcal{P}\in S} (N(\mathcal{P})+1), 
\end{equation}
where $m$ is an integer with $0\leq m \leq |S|$.



\section{Generic elements}\label{sec:gen el}
Let $H_{a,b}$ denote the product of $a$ copies of the hyperbolic plane and $b$ copies of the hyperbolic $3$-space.
The group $G = \PGL_2(\R)^a\times\PGL_2(\C)^b$ is the group of all isometries of $H_{a,b}$ which preserve each factor and the orientation of the three-dimensional factors (recall that $\PGL_2(\R)$ has two connected components one of which is isomorphic to $\PSL_2(\R)$, while $\PGL_2(\C)$ is isomorphic to $\PSL_2(\C)$). Given a subgroup $\Gamma \subset G$, the group $\Gamma^{(2)}$ generated by squares consists of orientation preserving isometries. 

Let $\Gamma$ be an arithmetic subgroups of $G$ defined over a number field $k$ and let $W\in \Gamma^{(2)}$. We denote by $\alpha = \alpha(W)$ an eigenvalue of the pre-image of $W$ in $\SL_2(k)$. Then $\alpha$ is well-defined up to multiplication by $\pm 1$, and we will refer to it as an eigenvalue of $W$.

\begin{definition}
Let $\Gamma \subset \PGL_2(\R)^a\times\PGL_2(\C)^b$ be an arithmetic subgroup defined over a field $k$ and let $W = (W_1, \ldots, W_{a+b})\in \Gamma^{(2)}$, $\alpha = \alpha(W)$. We call the element $W$ \emph{generic} if its components $W_1,\ldots,W_a$ are hyperbolic, components $W_{a+1},\ldots,W_{a+b}$ are loxodromic (and not hyperbolic), and the field $\Q(\alpha + 1/\alpha) = k$.
\end{definition}

This definition is related to the general notion of generic elements in semi-simple algebraic groups studied, in particular, by Prasad and Rapinchuk (see \cite[Section~9.4]{PR1} for the references and related discussion). In this paper we restrict our attention to the $\PGL_2$-case where the definition given above is sufficient. 

We recall the exceptional isomorphism 
$$\Psi: \PGL_2(\C) \to \SO^+_{3,1},$$
which is described in detail, for instance, in \cite[Section~1.3]{EGM-book} and \cite[Section~10.2]{MR-book}. 

The map $\Psi$ induces an embedding 
$$\Phi: \Gamma^{(2)} \to (\SO^+_{2,1})^a\times(\SO^+_{3,1})^b \subset \GL_{3a+4b}(\R).$$

The map $\Phi$ can be used to distinguish generic elements of $\Gamma$. To this end we notice that an element $M\in \SO^+_{3,1}$ is loxodromic if and only if the matrix $M$ has no eigenvalue equal to $1$. Similarly, an element of $\SO^+_{2,1}$ is hyperbolic if and only if its eigenvalue corresponding to the negative eigenvector is not equal to $1$ (see \cite[Chapter~2, Propositions~1.13--1.16]{EGM-book}).

\begin{lemma}[An algebraic criterion for genericity]\label{lemma-gen-cr}
There exists a proper real subvariety $X$ of $G = \PGL_2(\R)^a\times\PGL_2(\C)^b$ such that for any arithmetic subgroup $\Gamma \subset G$ all elements of $\Gamma^{(2)}$ which are not in $X$ are generic. 
\end{lemma}

\begin{proof}
The maps $\Psi$ and $\Phi$ considered above are $\R$-morphisms. We can define a subvariety $X_1$ as the Zariski closure of the pre-image in $G$ of the set of elements whose $\SO^+_{3,1}$-component has an eigenvalue equal to $1$. Similarly, define $X_2$ as the closure of the pre-image of the elements whose $\SO^+_{2,1}$-component has an eigenvalue corresponding to the negative eigenvector equal to $1$. 

It remains to exclude the elements for which the field $k_0 = \Q(\alpha + 1/\alpha)$ is a proper subfield of $k$. For every $i, j = 1,2, \dots, a +b$ such that $i \neq j$, let  $f_{i,j} := \tr(W_i) - \tr(W_j),$ this is an algebraic function different from zero in $G$, and similarly, for $j =1, \dots, b$, let $ g_j := \tr(W_{a +j}) - \tr(\overline{W_{a+j}})$. Let $X_3$ be the union of the zeros of all such algebraic functions $f_{i,j}, g_j$. Observe that if $W$ is in the complement of $X_3$, then $\alpha + 1/\alpha = \tr(W)$ has at least $n = a + 2b$ different embeddings in $\C$, and therefore $[k_0: \Q] = n$ and as $k_0 \subset k$, we must have $k_0 = k$.  In conclusion, we can take $X = X_1 \cup X_2 \cup X_3$.


\end{proof}

\begin{lemma}\label{lemma2}
Let $\Gamma \subset \PGL_2(\R)^a\times\PGL_2(\C)^b$ be an arithmetic subgroup defined over a field $k$ and $W \in \Gamma^{(2)}$ a generic element with an eigenvalue $\alpha$. Then the quaternion algebra $A(k)$ associated to $\Gamma$ splits over the field $\ell = \Q(\alpha)$. 
\end{lemma}

\begin{proof}
By \cite[Lemma 12.2.1]{MR-book}, the field $\ell$ embeds isomorphically as a subfield of $A(k)$. If $A(k)$ is a division algebra, then there exists a place $v \in V_\infty(k)$ for which $A(k_v)$ is ramified. It follows that $|\alpha_v + 1/\alpha_v| < 2$ and thus the extension of $k_v$ to $\ell$ is complex. So in this case $\ell$ is a quadratic extension of $k$ and we apply \cite[Theorems 12.2.3]{MR-book} to conclude that $A(k)$ splits over $\ell$. In the remaining case $A(k)$ already splits over $k$ hence it splits over $\ell \supseteq k$.
\end{proof}

\begin{example}\label{example-PSL(2,C)}
Let $G = \PGL_2(\C)$ ($=\PSL_2(\C)$). In this case the assumption that $W = W_1 \in \Gamma^{(2)}$ is loxodromic already implies that for its eigenvalue $\alpha$ the field $\Q(\alpha + 1/\alpha) = k$. Indeed, assume that $\Q(\alpha + 1/\alpha)$ is a proper subfield of $k$. Since $k$ has exactly one complex place, all its proper subfields are totally real, but by the definition of loxodromic elements we have that $\alpha$ is not real and not a unit. It follows that $\Q(\alpha + 1/\alpha) = k$ and $[\Q(\alpha):k] = 2$ (cf. \cite[Chapter~12]{MR-book}). 

Notice that here the subvariety $X$ of non-generic elements has a particularly simple form given by 
$ X = \{ M \in G \mid \mathrm{Det}(\Phi(M) - \mathrm{Id}) = 0\}$.
\end{example}



\section{Proof of the main theorem for \texorpdfstring{$G = \PGL_2(\C)$}{G = PGL(2,C)}.}

In this section we prove Theorem~\ref{thm1} for the arithmetic Kleinian groups. This argument already contains all the main ideas but allows to avoid some technicalities. 

\subsection{Generic elements}\label{sec:step1}
The first main ingredient of the proof is the following variant of Lemma~3.1 from \cite{CHL21}:

\begin{lemma}\label{lemma-word} 
Given a non-trivial word $w \in \mathrm{F}_2$ such that $w(A,B)\in \Gamma^{(2)}$ for any $A$, $B\in \Gamma$, there exists $n_w$ (only depending on $w$ and $G$) such that for any subset $F \subset \Gamma$ generating a Zariski dense over $\R$ subgroup of $G$, the element $w(A,B)$ is generic for some $A, B \in F^{n_w}$.
\end{lemma}

\begin{proof}
As in \cite{CHL21}, the proof relies on \cite[Proposition~3.2]{EMO05} or \cite[Lemma~4.2]{Breu11}. We consider $G$ as a real algebraic subgroup of $\GL_4(\R)$ defined by the map $\Phi$ from Section~\ref{sec:gen el}, and think of $G\times G$ as a subgroup of $\GL_8(\R)$ embedded diagonally. Let $G'$ be the Zariski closure of $\Gamma\times \Gamma$ in $\GL_8(\R)$. Then for the subset $X$ from Lemma~\ref{lemma-gen-cr}, we define
\begin{align*}
Y :=& \{ (A,B) \in G \times G:\ w(A,B) \in X \} \\
   =& \{ (A,B) \in G \times G:\ \mathrm{Det}(w(A,B) - \mathrm{Id}) = 0 \}.
\end{align*}
It follows that the set $Y\cap G'$ is a proper Zariski closed over $\R$ subset of $G'$.

Now by \cite[Proposition~3.2]{EMO05}, there exists $n_w \geq 1$ depending only on $w$ such that  
$Z = \{ (A,B)\in \Gamma\times \Gamma :\  A, B\in F^{n_w}\} \setminus Y$ is non-empty. By Lemma~\ref{lemma-gen-cr}, for $(A, B) \in Z \subset F^{n_w}\times F^{n_w}$ the element $w(A,B)$ is generic.
\end{proof}

\subsection{Almost laws}\label{sec:step2}
The next key step is to use almost laws. Recall that for a group $G$, a \emph{law} is a non-trivial element $w$ in the free group $\mathrm{F}_n$ such that the image of the associated word map $w_G: \prod_n G \to G$ is the identity $1_G$. Given $\epsilon  > 0$ and a metric $\rho$ on $G$, an \emph{$\epsilon$-almost law} is a non-trivial element $w \in \mathrm{F}_n$ such that the image of $G$ lies in an $\epsilon$-neighborhood of $1_G$. A result of A.~Thom \cite{Thom13}, which has also been attributed to E.~Lindenstrauss, shows that for any compact Lie group $G$ and any $\epsilon  > 0$ there exists an \emph{$\epsilon$-almost law} $w \in \mathrm{F}_2$ on $G$. We refer to \cite[Section~2.2]{CHL21} for some more details about laws and almost laws. 

Let $w_0$ be an almost law on $\mathrm{O}_4$, a maximal compact subgroup of $\GL_4(\R)$. For the later purposes we will modify $w_0$ by composing it with a double commutator. More precisely, let 
$$w(A,B) := w_0([A^2, [B^2, A^2]], A^{-1}[A^2, [B^2, A^2]]A).$$
As $[A^2, [B^2, A^2]]$ and its conjugate generate a free subgroup of $\mathrm{F}_2$, the word $w(A,B)$ is not trivial. Moreover, by definition its values on $\Gamma$ are in $\Gamma^{(2)}$. By Lemma~\ref{lemma-word}, there exist $n_w \geq 1$ and $A_0$, $B_0\in F^{n_w}$ such that $W = w(A_0, B_0)$ is a generic element. 

Let $\alpha$ be an eigenvalue of $W$ whose absolute value is bigger than $1$. It is defined up to the multiplication by $\pm1$ and we will choose the sign so that $\mathrm{Re}(\alpha) \ge 0$. Denote by $\ell = \Q(\alpha)$ the extension of $\Q$ generated by $\alpha$. The field $\ell$ is a quadratic extension of $k$, the field of definition of $\Gamma$, and the quaternion algebra $A(k)$ splits over $\ell$ (cf. Example~\ref{example-PSL(2,C)}). 

\subsection{Discriminant}\label{sec:step3}
Let $d = d_\ell = [\ell:\Q] = 2d_k$. Notice that since $k$ has at least one complex place, the degree $d_\ell \geq 4$. Consider the discriminant of the field $\ell$:
\begin{align*}
    \Delta_\ell &= \prod_{1\leq i < j \leq d} |\alpha_i - \alpha_j|,
\end{align*}
where $\alpha_1 = \alpha$ and $\alpha_i$ are its Galois conjugates. We denote by $\alpha_1 = \alpha$ and $\alpha_2 = 1/\alpha$ the eigenvalues of $W$, by $\alpha_3 = \bar{\alpha}$ and $\alpha_4 = 1/\bar{\alpha}$ the eigenvalues of $\overline{W}$, and by $\alpha_5, \ldots, \alpha_d$ the other conjugates. Because $w$ is an almost law we have that $\alpha_5, \ldots, \alpha_d$ are very close to $1$, say, $|\alpha_i - 1|\leq \epsilon <1$, $i = 5, \ldots, d$. 

Now 
\begin{align*}
\log \Delta_\ell = \sum_{1\leq i < j \leq d}\log |\alpha_i - \alpha_j|
= S_1 + S_2 + S_r,
\end{align*}
where we denote by $S_1$ the sum of the terms that contain $\alpha_1$ and 
$\alpha_3 = \bar{\alpha}_1$, by $S_2$ the sum of the terms that contain $\alpha_2$ and $\alpha_4 = \bar{\alpha}_2$ and do not have $\alpha_1$ and its conjugate, and by $S_r$ the sum of the remaining terms.

We analyze each of the three sums separately. 

Because $|\alpha_i-1|$,  $i =5, \ldots, d$ are small, the sum $S_r$ is negative (or zero, when $d=4$). From now on we will assume that $d \ge D_0$ is sufficiently large, the case of bounded degree is much simpler and we will deal with it later on. The sum $S_r$ has $(d-4)$ choose $2$ terms each bounded above by $-c = -c(\epsilon)<0$, hence we have 
\begin{align}\label{eq:Sr}
S_r \leq -c\frac{(d-4)(d-5)}{2}.
\end{align}

Now back to the product formula for the discriminant. The terms $|\alpha_i-\alpha_j|$, $5\leq i<j\leq d$, are small depending on $\epsilon$. As $\Delta_\ell$ is a positive integer, it implies that the terms with $\alpha_1$ and $\alpha_3 = \bar{\alpha}_1$ are large. By choosing $\epsilon > 0$ sufficiently small we can assume that $|\alpha_1| > 10$. Notice that it shows that the geodesic associated to $W$ in the quotient hyperbolic space cannot be short, its length is bounded below by a constant which grows when $\epsilon$ is chosen for the almost law tends to zero. 

From $|\alpha_1| > 10$ we have $|\alpha_2| = |1/\alpha_1| < 1/10$ and $|\alpha_4| = |1/\bar{\alpha}_1| <1/10$. Recalling that $\mathrm{Re}(\alpha) \ge 0$ and that $\alpha_i$,  $i =5, \ldots, d$, are close to $1$ and have absolute value $1$, we conclude that all the terms in the sum $S_2$ are negative, hence $S_2 < 0$. 

Therefore, we have
\begin{align}\label{eq:sum1}
	S_1 &\geq \log \Delta_\ell, \text{ and }\\
\label{eq:sum2}
	S_1 &\geq - S_r. 
\end{align}

Next, we would like to bound the sum $S_1$ from above. The following inequality for all $i = 2, \ldots, d$ easily follows from the information we have about $\alpha$: 
\begin{align*}
    \log|\alpha_1 - \alpha_i| &\leq \log|\alpha| + 2.
\end{align*}
The same applies to the terms with $\alpha_3 = \bar{\alpha}_1$. Thus we have 
\begin{align}\label{eq:sum3}
    4d\log|\alpha| &> S_1.
\end{align}

Now from \eqref{eq:sum1} we get
\begin{align*}
	4d \log |\alpha| & \geq \log \Delta_\ell;
\end{align*}
\begin{align}\label{eq:disc}
	\log |\alpha| & \geq \frac{1}{4d} \log \Delta_\ell. 
\end{align} 
On the other hand, inequalities \eqref{eq:Sr}, \eqref{eq:sum2} and \eqref{eq:sum3} imply
\begin{align*}
4d \log |\alpha| \geq c\frac{(d-4)(d-5)}{2},
\end{align*}
where $c = c(\epsilon)$ is a positive constant, hence we have
\begin{align}\label{eq:deg}
\log|\alpha| \geq c_1(d-4).
\end{align}



\subsection{Height and volume}\label{sec:step4}
Now from one hand we can relate $\log|\alpha|$ to the normalized height $\hhat(F)$, and from the other hand we can relate the discriminant $\Delta_\ell$ to the covolume. 

For the first inequality recall that $W = w(A_0, B_0)$ and $\alpha$ is its eigenvalue with $|\alpha| > 1$. By the definition and the basic properties of the normalized height in \cite[Propositions~2.13 and 2.14]{Breu11}, 
we have:
\begin{align}\label{eq-displ}
\hhat(F) = \frac{\hhat(F^n)}{n}
\geq \frac{h(\alpha)}{n} = \frac{2\log |\alpha|}{nd},
\end{align}
where $n = |w|n_w$ is the length of $W$ as a product of the elements of $F$.

In view of \eqref{eq:deg}, formula \eqref{eq-displ} already shows that $\hhat(F)$ is bounded below by a positive constant. It remains to relate it to the volume.

For the second inequality we use Borel's volume formula \eqref{formula:maxcovolume}. 

The field $\ell$ is a quadratic extension of $k$, hence we have
\begin{align*}
	\Delta_k^2 \Delta_{\ell/k} &= \Delta_\ell.
\end{align*}

By \eqref{eq:disc}, $c \log \Delta_\ell \le d\log |\alpha|$. Hence we get
\begin{align}\label{eq2}
	\Delta_k^2 \leq e^{c_1 d (\log|\alpha|)} \text{ and } 
	\Delta_{\ell/k} \leq e^{c_1d (\log|\alpha|)}, 
\end{align}
with the constant $c_1 = c_1(\epsilon) > 0$. 

Coming back to the volume formula \eqref{formula:maxcovolume}, we conclude that the minimal covolume in the commensurability class of $\Gamma$ satisfies
\begin{align*}
	\covol(\Gamma_0) \leq c_2e^{c_3d(\log|\alpha|)}.
\end{align*}
Here we used the basic facts that $\Delta_{\ell/k} = \prod_{\mathcal{P}|\Delta(A)} N(\mathcal{P})$ and $\zeta_k(2) \leq\zeta_{\Q}(2)^{d_k}$. 

Together with \eqref{eq-displ} it gives the bound
\begin{align*}
	\covol(\Gamma_0) \leq c_4e^{c_5\hhat(F)d^2}.
\end{align*}

If this bound fails, then the group generated by $F$ is not Zariski dense and we can not guarantee  the existence of the generic element $W$. 

\subsection{Action on the Bruhat--Tits tree}\label{sec:step5}
We now consider other maximal arithmetic subgroups commensurable with $\Gamma$ and containing $W$. Let $\Gamma_1$ be a maximal arithmetic subgroup associated to a finite subset $S \subset V_f(k)$ and assume that $\Gamma_1 \supset \Gamma$.
The idea is as follows, when we look at the action of $\Gamma_1$ on the Bruhat--Tits tree of $\PGL_2(k_v)$ for $v\in S$, it has to be in the stabilizer of an edge. Hence the subgroup $\Gamma'_1$ of index $2$ in $\Gamma_1$ fixes an edge of the tree. Now fixing an edge implies that the image of $\Gamma'_1$ in $\PGL_2(\mathbb{F}_q)$, $\mathbb{F}_q$ is the residue field of $k$ at $v$, lies in a solvable subgroup (as fixing an edge implies that one fixes a point in the projective line $\mathrm{P}^1(\mathbb{F}_q)$). In our case it is a solvable group of the derived length $2$, thus taking two commutators guarantees that the reduction in $\PGL_2(\mathbb{F}_q)$ is actually trivial. 

Recall that we defined $W$ as the word of $[A^2, [B^2, A^2]]$ and its conjugate. It follows that for all prime ideals $\mathcal{P}$ of $\mathcal{O}_k$ corresponding to the places in $S$ the trace of $W$, which is an algebraic integer of $k$, satisfies
\begin{equation}\label{cong4.5}
|\tr(W)| = \alpha + 1/\alpha \equiv 2 (\mathrm{mod}\ \mathcal{P}).
\end{equation}
Let $\beta = \alpha + 1/\alpha-2$. Congruences \eqref{cong4.5} and the 
product formula imply that 
$$|N_{k/\Q}(\beta)| \geq \prod_{\mathcal{P}\in S} N(\mathcal{P}).$$
Now $N(\mathcal{P}) \geq 2$ implies that $\# S \leq \log_2|N_{k/\Q}(\beta)|$, and hence
\begin{align*}
    2^{\log_2|N_{k/\Q}(\beta)|} |N_{k/\Q}(\beta)| &\geq \prod_{\mathcal{P}\in S} (N(\mathcal{P}) + 1);\\
    |N_{k/\Q}(\beta)|^2 &\geq \prod_{\mathcal{P}\in S} (N(\mathcal{P}) + 1).
\end{align*}

Recall that $\alpha$ is a unit in a quadratic extension $\ell$ of the field $k$ whose Galois conjugates satisfy that $\alpha_1 = \alpha$ and $\alpha_3 = \bar\alpha$ have absolute values bigger than $10$ and positive real part, while $i = 5, \ldots, d$, $|\alpha_i - 1|\leq \epsilon <1$ (cf. Section~\ref{sec:step3}). It follows that 
\begin{multline*}
|N_{k/\Q}(\beta)|^2 = |N_{\ell/\Q}(\beta)| = |N_{\ell/\Q}\left(\frac{(\alpha - 1)^2}{\alpha}\right)| = \\
|(\alpha - 1)(\bar\alpha - 1)(1/\alpha-1)(1/\bar\alpha-1)(\alpha_5-1)\ldots(\alpha_d-1)|^2 \leq c^2|\alpha|^4.
\end{multline*}

We conclude that the generalized index satisfies 
\begin{align*}
    [\Gamma_0:\Gamma_1] &\leq c_1|\alpha|^4, \text{ and}\\
    \covol(\Gamma_0) &\leq c_2e^{c_3 \hhat(F)d^2}, 
\end{align*}
with $c_2, c_3 > 0$ depending only on $G$.

If $W$ is contained in a congruence subgroup $\Gamma(\mathcal{I})\subset \Gamma_1$, then  
$W \equiv \id (\mathrm{mod}\ \mathcal{I})$, and hence we have $N(\mathcal{I}) \leq c_4 |N_{k/\Q}(\beta)|$ with $\beta = \alpha + 1/\alpha - 2 \in k$.
It follows that $[\Gamma_1:\Gamma(\mathcal{I})] \leq c_5|\alpha|^6$, where we used the facts that the  dimension of $\PGL_2$ is equal to $3$ and that  $|N_{k/\Q}(\beta)| \leq c|\alpha|^2$, which is proved as before. Therefore, the result applies to any congruence subgroup $\Gamma(\mathcal{I})$ containing $\Gamma$.


\medskip

This finishes the proof for $G= \PGL_2(\C)$. We leave as an exercise for the interested reader to adjust the argument for $\PGL_2(\R)$. The general case of the product $\PGL_2(\R)^a\times\PGL_2(\C)^b$ requires some more work, we consider the details in the next section.

\section{The general case of the main results}

Let $G = \PGL_2(\R)^a\times\PGL_2(\C)^b$. The proof of Theorem~\ref{thm1} follows the same steps as the proof for $G = \PGL_2(\C)$ presented in the previous section but requires some necessary modifications. We proceed with the technical details.

Lemma~\ref{lemma-word} and its proof carry out to the general case with the minimal changes: the statement remains the same, in the proof instead of $\GL_4(\R)$ we consider the real embedding of $G$ in $\GL_{3a+4b}(\R)$ and then apply Lemma~\ref{lemma2} and \cite[Proposition~3.2]{EMO05} to finish the proof.

The only modification in Section~\ref{sec:step2} is that now $w_0$ is an almost law on the group $(\mathrm{O}_3)^a \times (\mathrm{O}_4)^b$ considered as a compact subgroup of $\GL_{3a+4b}(\R)$.

The combinatorial part of estimating the sums in Section~\ref{sec:step3} requires significant modifications and we will redo it from scratch.
We have $d = d_\ell = [\ell:\Q] \leq 2d_k$. The discriminant 
\begin{align*}
    \Delta_\ell &= \prod_{1\leq i < j \leq d} |\alpha_i - \alpha_j|,
\end{align*}
where $\alpha_1,\ldots, \alpha_d$ are the Galois conjugates of an eigenvalue of $W$. By abuse of notation, we will assume that $\alpha = \alpha_j$ is the conjugate with the maximal absolute value.

Because $w$ is an almost law we have that $\alpha_{d_0}, \ldots, \alpha_d$  are very close to $1$ with $d_0 = 2a+4b+1$, say, $|\alpha_i - 1|\leq \epsilon <1$, $i = d_0, \ldots, d$. 

We write 
\begin{align*}
\log \Delta_\ell = \sum_{1\leq i < j \leq d}\log |\alpha_i - \alpha_j|
= S_1 + S_2 + S_r,
\end{align*}
This time we denote by $S_1$ the sum of all terms that contain $\alpha_i$ with $|\alpha_i| > 1$. Notice that because $W$ is a generic element, the eigenvalue $\alpha$ satisfies this condition and thus $S_1$ is non-empty. In $S_2$ we collect all the terms with $|\alpha_i| < 1$ and $|\alpha_j| \leq 1$ or $|\alpha_i| \leq 1$ and $|\alpha_j| < 1$; and in $S_r$ the remaining terms with $|\alpha_i| = |\alpha_j| = 1$.

We need to estimate the sums $S_1$, $S_2$, and $S_r$. Compared to \ref{sec:step3}, this time we will use more crude estimates. First let us assume that the degree of the field $\ell$ is large, $d \ge D_0 \gg 1$. When the degree is bounded the proof simplifies significantly as we will see later on. In the argument below $c_i$ will denote positive constants. We have:

The sum $S_1$ has $\sim d$ terms each bounded above by $\log|\alpha| + c_1$. As in Section~\ref{sec:step3}, the product formula for $\Delta_\ell$ and the fact that $\Delta_\ell \ge 1$ together with the condition that $|\alpha_i - \alpha_j|$ are small for $d_0 \le i<j \le d$ imply that we can assume that $\log|\alpha| \ge c(\epsilon) > 10$. This implies that $\log|\alpha| + c_1 < c_2\log|\alpha|$ and 
\begin{equation}\label{eq:51}
    c_3d\log|\alpha| > S_1.
\end{equation}

The sum $S_2$ also has $\sim d$ terms and for each of them we have
$$\log|\alpha_i - \alpha_j| = \log\left(\frac{|\alpha_i - \alpha_j|}{2}2\right) = \log\frac{|\alpha_i - \alpha_j|}{2} + \log 2 < \log 2,$$
hence
\begin{equation}\label{eq:52}
    S_2 < c_4d.
\end{equation}

The remaining $S_r$ has $\sim d^2$ summands, is negative, and satisfies
\begin{equation}\label{eq:53}
    S_r \leq -c_5d^2.
\end{equation}

We obtain 
\begin{align*}
S_1 + S_2 &< c_3d\log|\alpha| + c_4d \le c_6d\log|\alpha|;\\
S_1 + S_2 &\ge \log \Delta_\ell;\\
S_1 + S_2 &\ge S_r \ge c_5d^2.
\end{align*}
In conclusion
\begin{align*}
c_6d\log|\alpha| &\ge \log \Delta_\ell;\\
c_6d\log|\alpha| &\ge c_5d^2.
\end{align*}
As in \ref{sec:step3}, this leads to
\begin{align*}
\log|\alpha| &\ge c_7d;\\
\log|\alpha| &\ge c_8\frac{1}{d}\log\Delta_\ell.
\end{align*}


The remaining steps of the proof follow as in \ref{sec:step4}--\ref{sec:step5}. The only required modification is that in \ref{sec:step5} instead of $|N_{\ell/\Q}(\beta)|\leq c|\alpha|^2$ we now have $|N_{\ell/\Q}(\beta)|\leq c|\alpha|^{a+2b}$. 
\qed

\begin{proof}[Proof of Corollary~\ref{cor1}]
The argument is similar to the proof of the arithmetic Margulis lemma in \cite{FHR22}. 
Let $F = \{\gamma\in\Gamma \mid d(x, \gamma x) \leq \epsilon_G\log(\covol(\Gamma_1))^{\frac12}\}$. We will show that the subgroup generated by $F$ is not Zariski dense.

Assume for the sake of contradiction that $\langle F \rangle$ is $\R$-Zariski dense in $G$. Using a variant of \cite[Corollary~1.7]{Breu11} derived from Theorem~\ref{thm1}, we find an element $\gamma\in F^{N_1}$ with an eigenvalue $\lambda$ such that 

$$h(\lambda) \ge c \max (\frac{\log(\covol(\Gamma_1))}{[k: \Q]^2}, 1) \geq c_1  \frac{\log(\covol(\Gamma_1)^{\frac12}}{[k: \Q]},$$ 
and therefore
$[k:\Q]h(\lambda) \ge c_1\log(\covol(\Gamma_1))^{\frac12} =: V$, where $k$ is the field of definition of $\Gamma$. 
Next we relate $\lambda$ to the translation length $\ell(\gamma)$. To this end recall that for $v\in V_\infty(k)$ the corresponding translation length is given by 
\begin{equation}\label{eq:cor1}
\ell_v(\gamma) = 2\log^+|\lambda_v| + 2\log^+|1/\lambda_v| \text{ (cf. \cite[Lemma~12.1.1]{MR-book})}.
\end{equation}

Hence we have:
$$\ell(\gamma) = \sqrt{\sum_{v\in V_\infty}\ell_v(\gamma)^2} \geq \frac{1}{\sqrt{|V_{\infty}|}}\sum_{v\in V_\infty}\ell_v(\gamma) \geq \frac{1}{\sqrt{|V_{\infty}|}}[k:\Q]h(\lambda),$$
where the first inequality follows from the Cauchy--Schwarz inequality, and the second is a consequence of \eqref{eq:cor1} and the definition of height.

It follows that 
$$\ell(\gamma) \geq a_X V,$$
for a constant $a_X > 0$ depending on the symmetric space $X$.

Therefore, 
$$\epsilon_G \ge \frac{a_X V}{N_1} = \frac{a_X c_1\log(\covol(\Gamma_1))^{\frac12}}{N_1}.$$ 
This proves that for $\epsilon_G$ small enough, the group generated by $F$ is not Zariski dense.    
\end{proof}

\section{Two applications}\label{sec:appls}

\subsection{Finiteness of arithmetic maximal reflection groups}
The problem of finiteness of conjugacy classes of arithmetic maximal hyperbolic reflection groups in a given dimension $n$ goes back to now classical work of Vinberg and Nikulin in 1970's. It remained open for several decades and was solved independently by Nikulin \cite{Nikulin07_fin} and Agol, Belolipetsky, Storm, and Whyte \cite{ABSW08}. The proof in \cite{ABSW08} relies on deep results from representation theory which allow to effectively bound the spectral gap of the Laplacian on a congruence arithmetic quotient space. Nikulin's proof in \cite{Nikulin07_fin} is very short but is heavily based on his previous work \cite{Nikulin80, Nikulin81} and some results on the spectral gap. We refer to \cite{Bel16-survey} for a detailed review of these and related results. 

Recently, Fisher and Hurtado suggested a new proof of the finiteness theorem \cite{FH23}. It completely avoids spectral bounds from the theory of automorphic forms but still relies on \cite{Nikulin80} to show finiteness of admissible fields of definition of a fixed degree. The new ingredient of the proof in \cite{FH23} is the arithmetic Margulis lemma from \cite{FHR22}.

Strong arithmetic Margulis lemma~\ref{cor1} allows us to make the argument independent of \cite{Nikulin80}. The proof follows the same lines as in \cite{FH23} and we shall only briefly review it here. 
First assume the dimension $n=2$. Let $\Gamma$ be an arithmetic group generated by reflections in the sides of an acute-angled polygon $P$, which is a maximal reflection group.  Let $\mathcal{E} = \{e_1, e_2, \dots , e_k\}$ be the edges of $P$ and consider the collection of balls $ \mathcal{B} = \{B_1, B_2, \dots , B_k\}$ of radius $R =\frac{\epsilon_G}{2}\log(\covol(\Gamma_1))^{\frac12}$ centered at the midpoints of the edges of $P$,  where $\epsilon_G$ is the constant in Lemma~\ref{cor1} and $\Gamma_1$ is the maximal arithmetic group containing $\Gamma$ (it is known that in this case $\Gamma_1$ is the normalizer of $\Gamma$ in $\PGL_2(\R)$).
If a ball $B_i \in \mathcal{B}$ intersects three edges (or intersect two other balls in $\mathcal{B}$), the group generated by reflections in these three sides will be typically a non-virtually solvable group (unless the three edges are adjacent and meet at right angles, the case that requires some further considerations). Therefore, if $\covol(\Gamma_1)$ is sufficiently large, Lemma~\ref{cor1} implies that the balls in $B$ typically cannot have triple intersections and intersect at most two edges of $P$. To illustrate the argument, assume there are no triple intersections of balls of $\mathcal{B}$ and that each ball in $\mathcal{B}$ intersects only one edge, and so half of each ball is contained in $P$. This implies that the area of $P$ has to be greater than $\frac{k}{4} \vol(B_{\HH^2}(R))$, but elementary hyperbolic geometry shows that $P$ has volume at most $(k-2)\pi$. If $R$ is sufficiently large we arrive at a contradiction. Therefore, covolume of $\Gamma_1$ is bounded above by a constant. By Borel's theorem \cite{Bor81} there are only finitely many conjugacy classes of such groups $\Gamma_1$, and hence there are only finitely many  conjugacy classes of arithmetic maximal reflection groups $\Gamma$. In the case $n \geq 3$ we can apply a similar argument to a two-dimensional face of $P$. In this case we need to recall some basic properties of hyperbolic polyhedra due to Andreev and Vinberg. We refer to \cite{FH23} for more details.

Notice that this application only requires the strong arithmetic Margulis lemma for $G = \PGL_2(\R)$.

\subsection{Benjamini--Schramm convergence of quotients of hyperbolic spaces}

In this subsection we give a short proof of Theorem~A from \cite{FR19} which does not rely on the estimates of orbital integrals from \cite{Fra21}:
\begin{theorem}
If $G = \PGL_2(\R)$ or $\PGL_2(\C)$ and $\Gamma_n$ is a sequence of arithmetic lattices in $G$, which are either all congruence and pairwise distinct, or pairwise non-commensurable, then the sequence of locally symmetric spaces $X_n = \Gamma_n\backslash X$, $X = \HH^2$ or $\HH^3$, converges in the Benjamini--Schramm sense to $X$.
\end{theorem}

Recall that a sequence $(X_n)_{n\in\N}$ \emph{Benjamini--Schramm converges} to $X$ if for every $R > 0$ we have
$$\lim_{n\to\infty} \frac{\vol((X_n)_{\le R})}{\vol(X_n)} = 0,$$
where $(X_n)_{\le R}$ denotes the subset of points of $X_n$ at which the injectivity radius is smaller than $R$. Intuitively this means that the injectivity radius around a typical point of $X_n$ gets very large as $n$ goes to infinity.

\begin{proof}
Let first $G = \PGL_2(\C)$. The strategy of the proof is similar to \cite[Section 8.3]{FHR22} where the authors establish a weaker result for a much wider class of groups $G$ using the arithmetic Margulis lemma. By \cite[Proposition~3.2]{7S}, a sequence of locally symmetric spaces $\Gamma_n\backslash X$ Benjamini--Schramm converges to $X$ if and only if the sequence of measures $\mu_{\Gamma_n}$ (also called \emph{invariant random subgroups} or \emph{IRS}) converges to $\delta(1)$ in weak-* topology. We refer to \cite[Sections~2-3]{7S} or \cite[Section~8]{FHR22} for the definitions and basic properties of the invariant random subgroups in the relation with the Benjamini--Schramm convergence. 
Since the space of IRS on $G$ is compact it suffices to show that any accumulation point of a sequence $(\mu_{\Gamma_n})$ is equal to $\delta(1)$. We will show that the strong arithmetic Margulis lemma forces any limit to be supported on a proper Zariski closed subgroup of $G$. Then Zariski density for IRS (cf. \cite[Theorem~2.9]{7S}) implies that such a limit must be concentrated on the trivial subgroup. 

Suppose to the contrary that $\mu_{\Gamma_n}$ accumulates at a non-trivial IRS $\mu$. This means that there exists a closed subgroup $\Lambda \subset G$, $\Lambda \neq \{1\}$ in the support of $\mu$, so that $\mu(V) > 0$ for any Chabauty-neighbourhood $V$ of $\Lambda$. It follows that $V$ intersects the conjugacy class of $\Gamma_n$ for infinitely many $n$. In particular, $\Lambda$ is a limit of a sequence $g_n \Gamma_{k_n} g_n^{-1}$ in the Chabauty topology. We will deduce from this and the strong arithmetic Margulis lemma that $\Lambda$ must be a proper subgroup, which contradicts the Borel density for IRS (see Theorem~2.9 in \cite{7S}).

To simplify notation assume that the sequence $\Gamma_n$ actually converges to $\Lambda$. Let $\epsilon_G$ be the constant given by the strong arithmetic Margulis lemma~\ref{cor1}. We fix the radius $R>0$. Let $\Gamma_{1,n}\supseteq\Gamma_n$ be a sequence of congruence subgroups of $G$ containing $\Gamma_n$. By Borel's theorem \cite{Bor81}, we have $\covol(\Gamma_{1,n}) \to +\infty$, hence we can assume that $\epsilon_g\log(\covol(\Gamma_{1,n}))^{\frac12} > R + 1$ for all $n$. Let
\begin{equation}\label{eq61}
\Lambda_{n,R} = \langle \Gamma_n\cap \mathrm{B}(R+1)\rangle,\ \Lambda_R = \langle \Lambda \cap \mathrm{B}(R)\rangle.
\end{equation}
By Lemma~\ref{cor1} and Remark~\ref{rem2}, we have two possible cases: either the group $\Lambda_{n,R}$ is virtually abelian or it is a Fuchsian group. By passing again to a subsequence, if necessary, we can assume that all the groups $\Lambda_{n,R}$ are of the same type.

In the first case we recall that virtually abelian subgroups of $\PGL_2(\C)$ are:
\begin{enumerate}
    \item the finite groups of symmetries of the regular solids;
    \item infinite cyclic and dihedral groups and their finite extensions;
    \item doubly periodic groups of translations and finite extensions. 
\end{enumerate}
Moreover, it is well known that in each of the cases a virtually abelian subgroup contains an abelian subgroup of bounded index. Hence there exist $A > 0$ and a sequence $\Theta_{n,R} \subset \Lambda_{n,R}$ of abelian subgroups such that $[\Lambda_{n,R} : \Theta_{n,R}] \leq A$ for all $n$. The sequence $\Theta_{n,R}$ converges to an abelian subgroup $\Theta_{R} \subset \Lambda$, moreover, since any element of $\Lambda\cap \mathrm{B}(R)$ is a limit of a sequence $\gamma_n \in \Gamma_n \cap \mathrm{B}(R+1)$, we see that $\Lambda_{R} \cap \Theta_{R}$ has index at most $A$ in $\Lambda_{R}$. It shows that the set $\mathcal{A}_R$ of abelian subgroups contained in $\Lambda_{R}$ with index at most $A$ is nonempty. Each set $\mathcal{A}_R$ is finite since $\Lambda_{R}$ is finitely generated and whenever $R < R'$ and $\Theta \in \mathcal{A}_{R'}$ we have $\Theta\cap\Lambda_{R} \in \mathcal{A}_R$, so we can pick a $\Theta_1 \in \mathcal{A}_1$ such that for infinitely many $N \in \N$ there exists a $\Theta_{N} \in \mathcal{A}_N$ with $\Theta_{1} \subset \Theta_{N}$. We can then pick an integer $N_2 > 1$ and $\Theta_{N_2} \in \mathcal{A}_{N_2}$ such that $\Theta_{1} \subset \Theta_{N_2}$ and for infinitely many $N \in\N$ there exists a $\Theta_{N} \in \mathcal{A}_N$ with $\Theta_{N_2} \subset \Theta_{N}$. This way we construct an increasing sequence $N_i \in \N$ and $\Theta_{N_i} \in \mathcal{A}_{N_i}$ which satisfy $\Theta_{N_i} \subset \Theta_{N_i+1}$. It follows that $\Theta = \bigcup_{i\in\N} \Theta_{N_i}$ is an abelian subgroup of $\Lambda = \bigcup_{i\in\N} \Lambda_{N_i}$ of index at most $A$, which contradicts the Borel density for IRS.

In the second case we have $h_n \Lambda_{n,R} h_n^{-1} \subset \PGL_2(\R)$. This implies that the elements of $\Lambda_{n,R}$ have real traces. As any element of $\Lambda\cap \mathrm{B}(R)$ is a limit of a sequence $\gamma_n \in \Gamma_n \cap \mathrm{B}(R+1)$, we have that all elements of $\Lambda_R$ generated by this set have real traces. Considering a sequence $R_i \to +\infty$, we conclude that $\Lambda = \bigcup_{i\in\N} \Lambda_{R_i}$ consists of elements with real traces and hence is a non-trivial closed proper subgroup of the group $G$ considered as a real algebraic group. This again contradicts the Borel density for IRS and finishes the proof of the theorem for $G = \PGL_2(\C)$.

For the group $G = \PGL_2(\R)$ the proof is simpler. The same argument applies here but now the groups $\Lambda_{n,R}$ in \eqref{eq61} are all virtually abelian with an abelian subgroup of index at most $2$.
\end{proof}

In this application we used the strong arithmetic Margulis lemma for $G = \PGL_2(\R)$ and $\PGL_2(\C)$.

\bibliography{biblio.bib}{}
\bibliographystyle{siam}

\end{document}